\documentclass[12pt]{amsart}

\usepackage[utf8]{inputenc}

\usepackage[a4paper,nomarginpar]{geometry}
\usepackage[english]{babel}
\usepackage{enumitem}
\usepackage{amsmath,amsfonts,amssymb,amsthm}
\usepackage{amscd}
\usepackage{tikz}

\allowdisplaybreaks
\makeatletter
\def\subsection{\@startsection{subsection}{2}%
	\z@{.5\linespacing\@plus.7\linespacing}{.3\linespacing}%
	{\normalfont\bfseries}}
\makeatother
\theoremstyle{theorem}
\newtheorem{theorem}{Theorem}[section]
\newtheorem{lemma}[theorem]{Lemma}
\newtheorem{proposition}[theorem]{Proposition}
\newtheorem{corollary}[theorem]{Corollary}

\theoremstyle{definition}

\newtheorem{example}[theorem]{Example}
\newtheorem{remark}[theorem]{Remark}

\theoremstyle{remark} \theoremstyle{question} \theoremstyle{example}

\newcommand{\N}{\mathbb{N}}
\newcommand{\Z}{\mathbb{Z}}

\newcommand{\cB}{\mathcal{B}}

\newcommand{\cP}{\mathcal{P}}

\newcommand{\eps}{\varepsilon}

\DeclareRobustCommand{\rchi}{{\mathpalette\irchi\relax}}
\newcommand{\irchi}[2]{\raisebox{\depth}{$#1\chi$}}

\newcommand{\q}[2]{\mathcal{Q} ({#1},{#2})}
\newcommand{\qpos}[2]{\mathcal{Q}_{+}({#1},{#2})}
\newcommand{\qneg}[2]{\mathcal{Q}_{-}({#1},{#2})}
\begin{document}

\title[Li-Yorke Chaos for Composition Operators]{Li-Yorke Chaos for Composition Operators on $L^p$-Spaces}


%    author one information
\author{N. C. Bernardes Jr.}
\address{Departamento de Matem\'atica Aplicada, Instituto de
	Matem\'atica, Universidade Federal do Rio de Janeiro, Caixa Postal
	68530, Rio de Janeiro, RJ, 21945-970, Brazil.}
\curraddr{}
\email{ncbernardesjr@gmail.com}
\thanks{}

%    author two information
\author{U. B. Darji}
\address{Department of Mathematics, University of Louisville, Louisville,
	KY 40292, USA.}
\curraddr{}
\email{ubdarj01@gmail.com}
\thanks{}
%    author three information
\author{B. Pires}
\address{Departamento de Computa\c c\~ao e Matem\'atica, Faculdade de Filosofia,
	Ci\^encias e Letras, Universidade de S\~ao Paulo, Ribeir\~ao Preto, SP,
	14040-901, Brazil.}
\curraddr{}
\email{e-mail:benito@usp.br}
\thanks{}

%    \subjclass is required.
\date{}

\dedicatory{}

\begin{abstract}
Li-Yorke chaos is a popular and well-studied notion of chaos. Several simple
and useful characterizations of this notion of chaos in the setting of linear
dynamics were obtained recently. In this note we show that even simpler and
more useful characterizations of Li-Yorke chaos can be given in the special
setting of composition operators on $L^p$ spaces. As a consequence we obtain
a simple characterization of weighted shifts which are Li-Yorke chaotic.
We give numerous examples to show that our results are sharp.
\end{abstract}

\maketitle

%%%%%%%%%%%%%%%%%%%%%%%%%%%%%%%%%%%%%%%%%%%%%%%%%%%%%%%%%%%%%%%%%%%%%%%%%%%%%

\section{Introduction}

Throughout this note, $(X,\cB,\mu)$ will denote a measure space with
$\mu(X) \neq 0$ and $f : X \to X$ will be a {\it bimeasurable map}
(that is, $f(B) \in \cB$ and $f^{-1}(B) \in \cB$ for every $B \in \cB$)
for which there exists a constant $c > 0$ such that
\begin{equation}\label{condition}
\mu\big(f(B)\big) \geq c\mu(B) \ \textrm{ for every } B \in \cB.
\end{equation}
Condition (\ref{condition}) ensures that the composition operator
$T_f : \varphi \mapsto \varphi \circ f$ is a continuous linear operator
acting on $L^p(X,\cB,\mu)$ ($1 \leq p < \infty$). This constitutes a natural
class of operators. The topological transitivity and mixing properties of
this class of operators were investigated in the recent paper \cite{BDP}.
Our goal here is to investigate the notion of Li-Yorke chaos and some of
its variations for this class of operators. We will present several
characterizations and counterexamples.

For a broad view of the area of linear dynamics, we refer the reader to
the books \cite{BM,grosse}, to the more recent papers
\cite{BBMP,BBMP2,BCDMP,BMPP,GMM,M}, and to the references therein.

Let us recall that a continuous self-map $g$ of a metric space $(M,d)$
is said to be {\em Li-Yorke chaotic} if there exists an uncountable set
$S \subset M$ (called a {\em scrambled set} for $g$) such that each pair
$(x,y)$ of distinct points in $S$ is a {\em Li-Yorke pair} for $g$,
in the sense that
$$
\liminf_{n \to \infty} d(g^n(x),g^n(y)) = 0
\ \ \text{ and } \ \
\limsup_{n \to \infty} d(g^n(x),g^n(y)) > 0.
$$
In the case in which $S$ can be chosen to be dense (respectively, residual)
in $M$, we say that $g$ is {\em densely} (respectively, {\em generically})
{\em Li-Yorke chaotic}. This notion of chaos was introduced in \cite{LY70}
in the context of interval maps. It is among the most popular and well
studied notions of chaos.

Li-Yorke chaotic linear operators were investigated in \cite{NBBMP,BBMP2}.
In particular, it was shown that for any continuous linear operator $T$ on
 any Banach space $Y$, the following assertions are equivalent:
\begin{itemize}
\item $T$ is Li-Yorke chaotic;
\item $T$ admits a {\em semi-irregular vector}, that is, a vector $y \in Y$
      such that
      \begin{equation}\label{ly2}
      \liminf_{n \to \infty} \|T^n y\| = 0
      \ \ \text{ and } \ \
      \limsup_{n \to \infty} \|T^n y\| > 0;
      \end{equation}
\item $T$ admits an {\em irregular vector}, that is, a vector $z \in Y$
      such that
      \begin{equation}\label{ly3}
      \liminf_{n \to \infty} \|T^n z\| = 0
      \ \ \text{ and } \ \
      \limsup_{n \to \infty} \|T^n z\| = \infty.
      \end{equation}
\end{itemize}
Moreover, characterizations for dense Li-Yorke chaos and for generic Li-Yorke
chaos were also obtained in \cite{BBMP2}.

Our first result is a necessary and sufficient condition for the composition
operator $T_f$ to be Li-Yorke chaotic. It holds without any additional
condition on $\mu$ or $f$.

\begin{theorem}\label{thm2}
The composition operator $T_f$ is Li-Yorke chaotic if and only if there are
an increasing sequence $(\alpha_j)_{j \in \N}$ of positive integers and
a nonempty countable family $(B_i)_{i \in I}$ of measurable sets of
finite positive $\mu$-measure such that:
\begin{itemize}
\item [\rm (A)] $\displaystyle \lim_{j \to \infty}
      \mu\big(f^{-\alpha_j}(B_i)\big) = 0$ for all $i \in I$,
\item [\rm (B)] $\displaystyle \sup\Bigg\{
      \frac{\mu\big(f^{-n}(B_i)\big)}{\mu(B_i)}
      : i \in I, n \in \N\Bigg\} = \infty$.
\end{itemize}
\end{theorem}

Below is a consequence of Theorem~\ref{thm2}.

\begin{corollary}\label{Cor1}
Assume $f$ injective. The composition operator $T_f$ is Li-Yorke chaotic
if there exists a measurable set $B$ of finite positive $\mu$-measure
such that:
\begin{itemize}
\item [\rm  (i)] $\displaystyle \liminf_{n \to \infty}
                                     \mu\big(f^{-n}(B)\big) = 0$,
\item [\rm (ii)] $\displaystyle
      \sup\Bigg\{\frac{\mu\big(f^n(B)\big)}{\mu\big(f^m(B)\big)}
      : n \in \Z, m \in I, n < m\Bigg\} = \infty$,
\end{itemize}
where $I = \{m \in \Z : 0 < \mu\big(f^m(B)\big) < \infty\}$.
\end{corollary}

We will see in Example~\ref{ExInjCor2} that the injectivity hypothesis is
essential in Corollary~\ref{Cor1}. If $\mu$ is finite, then the converse
of Corollary~\ref{Cor1} holds. This follows easily from $(LY5)$ in
Theorem~\ref{thm1} below. However, for infinite measures this converse
may fail (see Example ~\ref{InfMeasConvCor2}). As an application of this
corollary, we have the following result:

\begin{corollary}\label{cthm2}
Assume $X = \Z$, $\cB = \cP(\Z)$ (the power set of $\Z$),
$f : i \in \Z \mapsto i+1 \in \Z$ and $0 < \mu(\{k\}) < \infty$ for some
$k \in \Z$. Then, the composition operator $T_f$ on $L^p(\Z,\cP(\Z),\mu)$
is Li-Yorke chaotic if and only if the following conditions hold:
\begin{itemize}
\item [\rm (a)] $\displaystyle \liminf_{i \to -\infty} \mu(\{i\}) = 0$,
\item [\rm (b)] $\displaystyle \sup\Bigg\{\frac{\mu(\{i\})}{\mu(\{j\})}
      : i,j \in \Z, i < j, 0 < \mu(\{j\}) < \infty\Bigg\} = \infty$.
\end{itemize}
\end{corollary}

It is well-known that the chaotic operators described in Corollary~\ref{cthm2}
are topologically conjugate to weighted backward shifts on $\ell^p (\Z)$ with
weights $w_i = \frac{\mu(i)}{\mu(i+1)}\cdot$ For example, see Section 1.4 of
\cite{BM} for more informations and relevant definitions. As as a simple
consequence, we get the following result. A similar characterization for
one-sided backward shifts was given in \cite{NBBMP}.

\begin{corollary}\label{Wshifts}
Let $w = (w_n)_{n \in \Z}$ be a bounded sequence of positive reals.
Define $B_w : \ell^p(\Z) \to \ell^p(\Z)$ by $B_w(e_n) = w_ne_{n-1}$.
Then, $B_w$ is Li-Yorke chaotic if and only if
$\sup\{w_n \cdots w_m: n < m, n,m \in \Z\} = \infty.$
\end{corollary}

Now we impose some additional conditions on $\mu$ and $f$ in order to obtain
some simpler necessary and sufficient conditions for $T_f$ to be Li-Yorke
chaotic.

\begin{theorem}\label{thm1}
If $\mu$ is finite and $f$ is injective, then the following assertions
are equivalent:
\begin{itemize}
\item [$(LY1)$] $T_f$ is Li-Yorke chaotic;
\item [$(LY2)$] $\exists \varphi \in L^p(X,\cB,\mu) \mid \varphi \not\equiv 0$
                and $\displaystyle \liminf_{n \to \infty} \|T_f^n \varphi\| = 0$;
\item [$(LY3)$] $\exists B \in \cB \mid \mu(B) > 0$ and
                $\displaystyle \liminf_{n \to \infty} \mu\big(f^{-n}(B)\big) = 0$;
\item [$(LY4)$] $\exists B \in \cB \mid \mu(B) > 0$ and
                $\displaystyle \liminf_{n \to \infty} \mu\big(f^n(B)\big) = 0$;
\item [$(LY5)$] $\exists B \in \cB \mid \mu(B) > 0$,
                $\displaystyle \liminf_{n \to \infty} \mu\big(f^{-n}(B)\big) = 0$ and
                $\displaystyle \liminf_{n \to\infty} \mu\big(f^{n}(B)\big) = 0$;
\item [$(LY6)$] $\exists B \in \cB \mid \mu(B) > 0$,
                $\displaystyle \liminf_{n \to \infty} \mu\big(f^{-n}(B)\big) = 0$ and
                $\displaystyle \limsup_{n \to \infty} \mu\big(f^{-n}(B)\big) > 0$;
\item [$(LY7)$] $T_f$ admits a characteristic function as a semi-irregular
                vector.
\end{itemize}
\end{theorem}

\begin{remark}\label{rthm1}
It will become clear from the proof of Theorem~\ref{thm1} that if we add the hypothesis
that $\mu(B) < \infty$ in $(LY3)$--$(LY6)$, then the following implications
always hold (even if $\mu(X) = \infty$ or $f$ is not injective):
$$
(LY6) \Longleftrightarrow (LY7) \Longrightarrow (LY1) \Longrightarrow
(LY2) \Longleftrightarrow (LY3) \Longleftarrow (LY5) \Longrightarrow (LY4).
$$
Moreover, $(LY4) \Rightarrow (LY5)$ whenever $\mu$ is finite and
$(LY5) \Rightarrow (LY6)$ whenever $f$ is injective.
However, we will give a series of counterexamples in Section~\ref{Se}
showing that no other implication holds in general, even under the
assumption that $\mu$ is $\sigma$-finite.
\end{remark}

Recall that $f$ is said to be {\it bi-Lipschitz with respect to $\mu$} if
there exist constants $c_2 > c_1 > 0$ such that
\begin{equation}\label{bimeasurable}
c_1 \mu(B) \le \mu\big(f(B)\big) \le c_2 \mu(B)
\ \ \text{ for every } B \in \cB.
\end{equation}
If $f$ is bijective and $f^{-1}$ denotes its inverse, then this property is
equivalent to saying that both composition operators $T_f$ and $T_{f^{-1}}$
are well-defined and continuous on $L^p(X,\cB,\mu)$. In this case, note that
$T_{f^{-1}} = T_f^{-1}$.

As an immediate consequence of Theorem~\ref{thm1}, we have the following
result:

\begin{corollary}\label{cthm1}
If $\mu$ is finite and $f$ is bijective and bi-Lipschitz with respect to
$\mu$, then $T_f$ is Li-Yorke chaotic if and only if so is $T_f^{-1}$.
\end{corollary}

It is well-known that an invertible operator $T$ on a Banach space $Y$
can be Li-Yorke chaotic without $T^{-1}$ being Li-Yorke chaotic.
Actually, we will see in Example~\ref{inli} that the corollary above
is false without the hypothesis that $\mu$ is finite.

Concerning the notions of dense Li-Yorke chaos and generic Li-Yorke chaos,
we have the results below.

\begin{proposition}\label{DenseLY}
If $\mu$ is finite and $f$ is injective, then the composition operator
$T_f$ is densely Li-Yorke chaotic if and only if it is topologically
transitive.
\end{proposition}

\begin{proposition}\label{GenericLY}
If $\mu$ is finite, then the composition operator $T_f$ is not generically
Li-Yorke chaotic.
\end{proposition}

\begin{remark}
We will see in Section~\ref{Se4} that there exist generically Li-Yorke
chaotic composition operators that are not topologically transitive.
This shows that we cannot remove the hypothesis that $\mu$ is finite
in the previous propositions.
\end{remark}

The proofs of the previous results will be given in the next section.
In the case of Theorem~\ref{thm1}, a key role will be played by the notions
of backward weakly wandering set and forward weakly wandering set~\cite{HaKa}.
In Section~\ref{Se} we will present several counterexamples.

%%%%%%%%%%%%%%%%%%%%%%%%%%%%%%%%%%%%%%%%%%%%%%%%%%%%%%%%%%%%%%%%%%%%%%%%%%%%%

\section{Proofs of the main results}\label{proofs}

%%%%%%%%%%%%%%%%%%%%

\subsection*{Proof of Theorem \ref{thm2}}

Assume $T_f$ Li-Yorke chaotic and let $\varphi \in L^p(X,\cB,\mu)$ be an
irregular vector for $T_f$. Consider the measurable sets
$$
B_i = \{x \in X : 2^{i-1} \leq |\varphi(x)| < 2^i\} \ \ \ \ (i \in \Z)
$$
and let
$$
I = \{i \in \Z : \mu(B_i) > 0\}.
$$
Because $\sum_{i \in \Z} 2^{(i-1)p} \mu(B_i) \leq \int_X |\varphi|^p d\mu
< \infty$, we have that $0 < \mu(B_i) < \infty$ for all $i \in \Z$.
Since $\varphi$ is an irregular vector for $T_f$, there is an increasing
sequence $(\alpha_j)_{j \in \N}$ of positive integers such that
$\lim_{j \to \infty} \|T_f^{\alpha_j} \varphi\| = 0$. This implies (A).
Now, suppose that (B) is false. Then, there is a constant $C < \infty$
such that
$$
\mu\big(f^{-n}(B_i)\big) \leq C \mu(B_i) \ \ \text{ whenever } i \in \Z
  \text{ and } n \in \N.
$$
Hence, for each $n \in \N$,
\begin{align*}
\|T_f^n \varphi\|^p &= \sum_{i \in \Z} \int_{f^{-n}(B_i)}
                       |\varphi \circ f^n|^p d\mu
                     \leq \sum_{i \in \Z} 2^{ip} \mu\big(f^{-n}(B_i)\big)\\
                    &\leq 2^p C \sum_{i \in \Z} 2^{(i-1)p} \mu(B_i)
                     \leq 2^p C \|\varphi\|^p.
\end{align*}
This contradicts the fact that the $T_f$-orbit of $\varphi$ is unbounded.

Let us now prove the converse. Let $Y$ be the closed linear span of
$\{\chi_{B_i} : i \in I\}$ in $L^p(X,\cB,\mu)$. It follows from (A)
that the set $R_1$ of all vectors $\varphi$ in $Y$ whose $T_f$-orbit
has a subsequence converging to zero is residual in $Y$.
For each $i \in I$, let
$$
\phi_i = \frac{1}{\mu(B_i)^\frac{1}{p}} \cdot \chi_{B_i} \in Y.
$$
Then
$$
\|\phi_i\| = 1 \ \ \text{ and } \ \
\|T_f^n \phi_i\|^p = \frac{\mu\big(f^{-n}(B_i)\big)}{\mu(B_i)}
\ \ \ (n \in \N).
$$
Therefore, (B) gives $\sup_{n \in \N} \|T_f^n|_Y\| = \infty$. Hence,
by the Banach-Steinhaus Theorem \cite[Theorem~2.5]{R}, the set $R_2$
of all vectors $\varphi$ in $Y$ whose $T_f$-orbit is unbounded is residual
in $Y$. Since each $\varphi \in R_1 \cap R_2$ is an irregular vector for
$T_f$, we conclude that $T_f$ is Li-Yorke chaotic.

%%%%%%%%%%%%%%%%%%%%

\subsection*{Proof of Corollary \ref{Cor1}}
Suppose that there exists such a set $B$. If
$\limsup_{n \to \infty} \mu\big(f^{-n}(B)\big) \neq 0$,
then $\varphi =\chi_B$ is a semi-irregular vector for $T_f$ because of (i),
implying that $T_f$ is Li-Yorke chaotic. Otherwise,
$\lim_{n \to \infty} \mu\big(f^{-n}(B)\big) = 0$.
Set $B_i = f^i(B)$ for each $i \in \Z$. Since $f$ is injective, condition (A)
of Theorem~\ref{thm2} is satisfied. For any $n < m$, we have that
$B_n = f^n(B) = f^{n-m}(f^m(B)) = f^{n-m}(B_m)$, which gives
\begin{align*}
\sup\Bigg\{\frac{\mu\big(f^n(B)\big)}{\mu\big(f^m(B)\big)}
           : n \in \Z, m \in I, n < m\Bigg\}
&= \sup\Bigg\{\frac{\mu\big(f^{n-m}(B_m)\big)}{\mu(B_m)}
                 : n \in \Z, m \in I, n < m\Bigg\}\\
&= \sup\Bigg\{\frac{\mu\big(f^{-n}(B_i)\big)}{\mu(B_i)}
                 : i \in I, n \in \N\Bigg\}.
\end{align*}
In this way, (ii) implies condition (B) of Theorem~\ref{thm2}.
Thus, $T_f$ is Li-Yorke chaotic.

%%%%%%%%%%%%%%%%%%%%

\subsection*{Proof of Corollary \ref{cthm2}}

For the sufficiency of the conditions, it is enough to choose $k \in \Z$
such that $0 < \mu(\{k\}) < \infty$ and to apply Corollary~\ref{Cor1}
with $B = \{k\}$.

For the necessity of the conditions, note that (a) follows from the fact
that $(LY1)$ always implies $(LY3)$ (Remark~\ref{rthm1}). If (b) is false,
then there is a constant $C \in (0,\infty)$ with
$$
\mu(\{i\}) \leq C \mu(\{j\}) \ \ \text{ whenever } i,j \in \Z \text{ and }
  i < j.
$$
Hence, for every $\varphi \in L^p(\Z,\cP(\Z),\mu)$ and every $n \in \N$,
\begin{align*}
\|T_f^n \varphi\|^p &= \sum_{i \in \Z} |\varphi(i+n)|^p \mu(\{i\})
                     = \sum_{i \in \Z} |\varphi(i)|^p \mu(\{i-n\})\\
                    &\leq C \sum_{i \in \Z} |\varphi(i)|^p \mu(\{i\})
                     = C \|\varphi\|^p.
\end{align*}
Thus, all orbits under $T_f$ are bounded, contradicting the fact
that $T_f$ is Li-Yorke chaotic.

%%%%%%%%%%%%%%%%%%%%

\subsection*{Proof of Theorem \ref{thm1}}\label{Proof}

In order to prove this theorem, we need the concept of weakly wandering set
and a couple of basic lemmas. The concept was first defined in \cite{HaKa}.
Lemma~4 of that paper is analogous to the lemmas we prove here. However,
our hypotheses on $f$ are different and hence that lemma does not apply
directly to our situation. Hence, we include the proofs.

We say that a measurable set $W$ is a {\it backward weakly wandering set}
for $f$ if there exists a sequence of positive integers
$k_1 < k_2 < k_3 < \cdots$ such that the measurable sets
$
\{W,f^{-k_1}(W),f^{-k_2}(W),f^{-k_3}(W),\ldots\}
$
are pairwise disjont. Likewise, if the measurable sets
$
\{W,f^{k_1}(W),f^{k_2}(W),f^{k_3}(W),\ldots\}
$
are pairwise disjoint, then we say that $W$ is a {\it forward weakly
wandering set} for $f$.

\begin{lemma}\label{wandering}
	If there exists $B \in \cB$ such that
	\begin{equation}\label{xd3}
	\mu(B) > 0
	\ \ \text{ and } \ \
	\liminf_{k \to \infty} \mu\big(f^k(B)\big) = 0,
	\end{equation}
	then $f$ admits a backward weakly wandering set $W \subset B$ of positive
	$\mu$-measure.
\end{lemma}

\begin{proof}
	Note that, by (\ref{condition}) and the second condition in (\ref{xd3}),
	$\mu(B)$ is necessarily finite. Let $\epsilon = \mu(B)/2$ and
	$\epsilon_i = \epsilon/(i\cdot 2^i)$ for all $i \ge 1$.
	By the second condition in (\ref{xd3}), we can construct a sequence
	$0 = k_0 < k_1 < k_2 < \cdots$ of non-negative integers such that
	\begin{equation}\label{maxxx}
	\max\left\{{c^{-r}} : 0 \le r \le k_{i-1}\right\} \cdot
	\mu\big(f^{k_{i}}(B)\big) < \epsilon_i \ \text{ for all } i \ge 1.
	\end{equation}
	By (\ref{condition}), $\mu\big(f^{-k}(S)\big) \le {c^{-k}} \mu(S)$
	whenever $k \ge 0$ and $S \in \cB$. This together with
	(\ref{maxxx}) yield, for all $0 \le j \le i-1$,
	\begin{equation}\label{33}
	\mu\left(f^{k_i - k_j}(B)\right)
	\le \mu\left(f^{-k_j}\big(f^{k_i}(B)\big)\right)
	\le {c^{-k_j}} \mu\left(f^{k_{i}}(B)\right)
	< \epsilon_i.
	\end{equation}
	We claim that
	$$
	W = B \Big\backslash \bigcup_{i=1}^\infty \bigcup_{j=0}^{i-1} f^{k_i-k_j}(B)
	$$
	is a backward weakly wandering set of positive $\mu$-measure.
	In fact, by (\ref{33}), we have that
	$$
	\mu(W)
	\ge \mu(B) - \mu\bigg(\bigcup_{i=1}^\infty \bigcup_{j=0}^{i-1} f^{k_i-k_j}(B)\bigg)
	\ge \mu(B) - \sum_{i=1}^\infty i \epsilon_i = \frac{\mu(B)}{2} > 0.
	$$
	Moreover, by the definition of $W$, for each $i \ge 1$ and $0 \le j \le i-1$,
	\begin{equation}\label{Wf}
	W \cap f^{k_i - k_j}(W) = \emptyset,
	\end{equation}
	and therefore
	$$
	\emptyset = f^{-k_i}\big(W \cap f^{k_i-k_j}(W)\big)
	= f^{-k_i}(W) \cap f^{-k_i}\big(f^{k_i-k_j}(W) \big)
	\supset f^{-k_i}(W) \cap f^{-k_j}(W).
	$$
	This proves that the sets $W = f^{-k_0}(W),f^{-k_1}(W),f^{-k_2}(W),\ldots$
	are pairwise disjoint, which means that $W$ is a backward weakly wandering
	set.
\end{proof}

\begin{lemma}\label{cwandering}
	Assume $f$ injective. If there exists $B \in \cB$ such that
	\begin{equation}\label{la1}
	0 < \mu(B) < \infty
	\ \ \text{ and } \ \
	\liminf_{k \to \infty} \mu\big(f^{-k}(B)\big) = 0,
	\end{equation}
	then $f$ admits a forward weakly wandering set $W \subset B$ of positive
	$\mu$-measure.
\end{lemma}

\begin{proof}
	Let $\epsilon = \mu(B)/2$ and $\epsilon_i = \epsilon/(i\cdot 2^i)$ for all
	$i \ge 1$. By the second condition in (\ref{la1}), there is a sequence
	$0 = k_0 < k_1 < k_2 < \cdots$ of non-negative integers such that
	$$
	\max\left\{{c^{-r}} : 0 \le r \le k_{i-1}\right\} \cdot
	\mu\big(f^{-(k_{i}-k_{i-1})}(B)\big) < \epsilon_i \ \text{ for all } i \ge 1.
	$$
	Thus, for all $0 \le j \le i-1$,
	\begin{align}\label{la3}
	\mu\left(f^{-(k_i - k_j)}(B)\right)
	&= \mu\left(f^{-(k_{i-1}-k_j)}\big(f^{-(k_i-k_{i-1})}(B)\big)\right)\notag\\
	& \le c^{-(k_{i-1}-k_j)} \mu\left(f^{-(k_i-k_{i-1})}(B)\right) < \epsilon_i.
	\end{align}
	We claim that
	$$
	W = B \Big\backslash \bigcup_{i=1}^\infty \bigcup_{j=0}^{i-1} f^{-(k_i-k_j)}(B)
	$$
	is a forward weakly wandering set of positive $\mu$-measure.
	In fact, by (\ref{la3}), we have that
	$$
	\mu(W)
	\ge \mu(B) - \mu\bigg(\bigcup_{i=1}^\infty \bigcup_{j=0}^{i-1} f^{-(k_i-k_j)}(B)\bigg)
	\ge \mu(B) - \sum_{i=1}^\infty i \epsilon_i = \frac{\mu(B)}{2} > 0.
	$$
	Moreover, by the definition of $W$, for each $i \ge 1$ and $0 \le j \le i-1$,
	$W \cap f^{-(k_i-k_j)}(W) = \emptyset$, which is equivalent to
	\begin{equation}\label{la4}
	f^{k_i - k_j}(W) \cap W = \emptyset.
	\end{equation}
	By the equality $f^{k_i}(W) = f^{k_j}\big(f^{k_i-k_j}(W)\big)$, by the
	injectivity of $f$ and by (\ref{la4}), we reach
	$$
	f^{k_i}(W) \cap f^{k_j}(W) = f^{k_j}\big(f^{k_i-k_j}(W)\big) \cap f^{k_j}(W)
	= f^{k_j}\big(f^{k_i-k_j}(W) \cap W\big) = \emptyset,
	$$
	which proves the claim.
\end{proof}

\begin{remark}\label{RW}
Given any $\delta > 0$, by replacing $\eps = \mu(B)/2$ by $\eps = \mu(B)/n$
with $n$ big enough in the proofs of Lemmas~\ref{wandering}
and~\ref{cwandering}, we see that the subset $W$ of $B$ can be chosen so that
$\mu(B \backslash W) < \delta$.
\end{remark}

Let us now prove Theorem~\ref{thm1}. We begin with the implications that
always hold (see Remark~\ref{rthm1}).

\medskip
\noindent $(LY1) \Rightarrow (LY2)$:
Suppose that $T_f$ is Li-Yorke chaotic. Then it admits a semi-irregular
vector $\varphi$ in $L^p(X,\cB,\mu)$. The second condition in (\ref{ly2})
yields $\varphi \not\equiv 0$.

\medskip
\noindent $(LY2) \Leftrightarrow (LY3)$:
Let $\varphi$ satisfy $(LY2)$. Hence, there exists $\delta > 0$ such that
the set $B = \left\{x \in X : |\varphi(x)| > \delta\right\}$ has positive
$\mu$-measure. Moreover,
\begin{equation}\label{ineq1}
\left\|T_f^k \varphi\right\|^p = \int \big|\varphi \circ f^k\big|^p \,d\mu
\ge \int_{f^{-k}(B)} \big|\varphi \circ f^k\big|^p \,d\mu
\ge \delta^p \mu\big(f^{-k}(B)\big).
\end{equation}
By $(LY2)$ and (\ref{ineq1}),
$\displaystyle \liminf_{k \to \infty} \mu\big(f^{-k}(B)\big) = 0$.
To prove the converse, just take $\varphi = \rchi_B$.

\medskip
\noindent $(LY5) \Rightarrow (LY3)$ and $(LY4)$:
Obvious.

\medskip
\noindent $(LY6) \Leftrightarrow (LY7)$:
Let $\varphi = \rchi_B$, where $B \in \cB$. Since
$$
\left\|T_f^k\varphi\right\|^p = \int \big|\varphi \circ f^k \big|^p \,d\mu
= \int \big|\rchi_B\big|^p \circ f^k \,d\mu
= \mu\big(f^{-k}(B)\big),
$$
it is clear that $(LY6)$ and $(LY7)$ are equivalent properties.

\medskip
\noindent $(LY7) \Rightarrow (LY1)$:
The existence of a semi-irregular vector, by itself, implies that
$T_f$ is Li-Yorke chaotic.

\medskip
The next implication uses the injectivity of $f$.

\medskip
\noindent $(LY5) \Rightarrow (LY6)$:
Let $B$ be as in $(LY5)$. By (\ref{condition}) and the third condition in
$(LY5)$, $\mu(B)$ is finite. The proof consists in constructing a measurable
set $A$ of positive $\mu$-measure and sequences
$n_1 < m_1 < n_2 < m_2 < \cdots$ of positive integers such that
\begin{equation}\label{que}
\lim_{k \to \infty} \mu\big(f^{-{m_k}}(A)\big) = 0
\ \ \text{ and } \ \
\limsup_{k \to \infty} \mu\big(f^{-n_k}(A)\big) > 0.
\end{equation}
Set $n_1 = 1$ and let $m_1 > n_1$ be any integer such that
\begin{equation}\label{max0}
\mu\big(f^{n_1-m_1}(B)\big) < \frac12\cdot
\end{equation}
Such $m_1$ exists because
$\displaystyle\liminf_{\ell \to \infty} \mu\big(f^{-\ell}(B)\big) = 0$.
Assume that $k \geq 2$ and that $n_1 < m_1 < n_2 < m_2 < \cdots
< n_{k-1} < m_{k-1}$ were defined. Let us define $n_k$ and $m_k$ as follows.
By hypothesis,
$\displaystyle\liminf_{\ell \to \infty} \mu\big(f^{\ell}(B)\big) = 0$.
Thus, there exists $n_k > m_{k-1}$ such that
\begin{equation}\label{max1}
\max\left\{{c^{-r}}: 0 \le r \le m_{k-1}\right\} \cdot
\mu\big(f^{n_{k}-m_1}(B)\big) < \dfrac{1}{2^{k}}\cdot
\end{equation}
Likewise, since
$\displaystyle\liminf_{\ell \to \infty} \mu\big(f^{-\ell}(B)\big) = 0$,
there exists $m_k > n_k$ such that
\begin{equation}\label{max2}
\max\left\{c^{-r}: 0 \le r \le n_k\right\} \cdot
\mu\big(f^{-(m_{k}-n_{k})}(B)\big) < \dfrac{1}{k \cdot 2^k}\cdot
\end{equation}
By induction, the infinite sequences $n_1 < m_1 < n_2 < m_2 < \cdots$
satisfy (\ref{max1}) and (\ref{max2}) for all $k \geq 2$.
Now, let us define the set $A$ so that (\ref{que}) is satisfied.
Set $A = \bigcup_{i=1}^\infty f^{n_i}(B)$. Then the second condition
in (\ref{que}) is automatic, because
$$
\mu\left(f^{-n_k}(A)\right) \geq \mu(B) \ \ \text{ for all } k \geq 1.
$$
Let us prove the first condition in (\ref{que}). By the injectivity of $f$,
$f^{-m_j}\big(f^{n_i}(B)\big) = f^{n_i-m_j}(B)$ for all $i,j \ge 1$. Hence,
\begin{equation}\label{Tfmk1}
\mu\left(f^{-m_j}(A)\right)
\le \sum_{i=1}^\infty \mu\big(f^{n_i-m_j}(B)\big)
  = \sum_{i=1}^{j}\mu\big(f^{n_i-m_j}(B)\big)+\sum_{i>j}\mu\big(f^{n_i-m_j}(B)\big).
\end{equation}
To find an upper bound for the sums in (\ref{Tfmk1}), we proceed as follows.
For all $i \ge 2$ and $1 \le j \le i-1$, we have that $0 \le m_j-m_1 \leq
m_{i-1}$. Hence, by (\ref{condition}), (\ref{max1}) and the injectivity
of $f$, for all $i\ge 2$ and $1\le j\le i-1$,
\begin{equation}\label{ad1}
\mu\left( f^{n_{i}-m_j}(B)\right)
  = \mu\left( f^{-(m_j-m_{1})}\left(f^{n_{i}-m_1}(B)\right)\right)
\le {c^{-(m_j-m_1)}} \mu\left(f^{n_{i}-m_1}(B)\right)
  < \dfrac{1}{2^{i}}\cdot
\end{equation}
In the same way, (\ref{max0}) and (\ref{max2}) yield, for all $j \ge 1$
and $1 \le i \le j$,
\begin{align}\label{ad2}
\mu\left(f^{n_i-m_j}(B)\right)
  &= \mu\left( f^{-(n_j-n_i)}\left( f^{-(m_j-n_j)}(B)\right)\right)\notag\\
  &\le {c^{-(n_j-n_i)}} \mu\left(f^{-(m_j-n_j)}(B)\right)
  < \dfrac{1}{j \cdot 2^j}\cdot
\end{align}
Hence,
$$
\sum_{i=1}^\infty \mu\big(f^{n_i-m_j}(B)\big)
\le \sum_{i=1}^j \dfrac{1}{j\cdot 2^j} + \sum_{i>j}\dfrac{1}{2^{i}}
  = \dfrac{1}{2^j} + \dfrac{1}{2^j}\cdot
$$
By (\ref{Tfmk1}),
$\displaystyle\lim_{j \to \infty} \mu\big(f^{-m_j}(A)\big) = 0$.

\medskip
The next implication requires the additional conditions that $\mu$ is finite
and $f$ is injective.

\medskip
\noindent $(LY3) \Rightarrow (LY4)$:
Let $B$ be as in $(LY3)$. By Lemma~\ref{cwandering}, there exist a measurable
set $W \subset B$ and a sequence of positive integers $k_1 < k_2 < k_3 < \cdots$
such that $\mu(W) > 0$ and the sets $W,f^{k_1}(W),f^{k_2}(W),\ldots$ are
pairwise disjoint. In particular,
$$
\sum_{i=1}^\infty \mu\big(f^{k_i}(W)\big)
  = \mu\left(\bigcup_{i=1}^\infty f^{k_i}(W)\right)
\le \mu(X)
  < \infty,
$$
implying that $\displaystyle\liminf_{k \to \infty} \mu\big(f^{k}(W)\big)=0$.

\medskip
The next implication requires the additional condition that $\mu$ is finite.

\medskip
\noindent $(LY4) \Rightarrow (LY5)$:
By Lemma \ref{wandering}, there are a measurable set $W \subset B$ and a
sequence of positive integers $k_1 < k_2 < k_3 < \cdots$ such that
$\mu(W) > 0$ and the sets $W,f^{-k_1}(W),f^{-k_2}(W),\ldots$ are pairwise
disjoint. In particular,
$$
\sum_{i=1}^\infty \mu\big(f^{-k_i}(W)\big)
  = \mu\left(\bigcup_{i=1}^\infty f^{-k_i}(W)\right)
\le \mu(X)
  < \infty,
$$
implying that $\displaystyle\liminf_{k \to \infty} \mu\big(f^{-k}(W)\big)=0$.
On the other hand, as $W \subset B$, we have by $(LY4)$ that
$\displaystyle\liminf_{k \to \infty} \mu\big(f^k(W)\big) = 0$.

%%%%%%%%%%%%%%%%%%%%

\subsection*{Proof of Proposition \ref{DenseLY}}

It was proved in \cite[Theorem~10]{BBMP2} that a continuous linear operator
$T$ on a separable Banach space $Y$ is densely Li-Yorke chaotic if and
only if it admits a dense set of irregular vectors (or a dense set of
semi-irregular vectors). It was also observed in \cite[Remark~22]{BBMP2}
that if an operator is topologically transitive, then it is densely
Li-Yorke chaotic. For the converse, assume $T_f$ densely Li-Yorke chaotic
and let $\eps \in (0,\min\{1,\mu(X)\})$. By the above-mentioned theorem
from \cite{BBMP2}, there is an irregular vector $\varphi$ for $T_f$ such that
$$
\|\varphi - \rchi_X\|^p < \eps^{p+1}.
$$
Set $B = \{x \in X : |\varphi(x) - 1| < \eps\}$.
Then $\mu(X \backslash B) < \eps$. Moreover,
$$
\|T_f^k \varphi\|^p \geq \int_{f^{-k}(B)} |\varphi \circ f^k|^p d\mu
\geq (1-\eps)^p \cdot \mu\big(f^{-k}(B)\big),
$$
that is,
$$
\mu\big(f^{-k}(B)\big) \leq \frac{1}{(1-\eps)^p} \cdot \|T_f^k \varphi\|^p.
$$
Since $\varphi$ is an irregular vector for $T_f$, this yields
$\liminf_{n \to \infty} \mu\big(f^{-n}(B)\big) = 0$.
By Lemma~\ref{cwandering} and Remark~\ref{RW}, there exists a measurable
set $W \subset B$ such that
$$
\mu(X \backslash W) < \eps \ \ \text{ and } \ \
\liminf_{n \to \infty} \mu\big(f^n(W)\big) = 0.
$$
By condition (C4) of \cite[Remark~2.1]{BDP}, $T_f$ is topologically
transitive.

%%%%%%%%%%%%%%%%%%%%

\subsection*{Proof of Proposition \ref{GenericLY}}

It was proved in \cite[Theorem~34]{BBMP2} that a continuous linear operator
$T$ on a Banach space $Y$ is generically Li-Yorke chaotic if and only if
every non-zero vector in $Y$ is semi-irregular for $T$. In our case, if
$\varphi = \rchi_X$ then $\|T_f^n \varphi\|^p = \mu(X)$ for all $n \geq 1$.
In particular, $\varphi$ is not a semi-irregular vector for $T_f$, and so
$T_f$ is not generically Li-Yorke chaotic.

%%%%%%%%%%%%%%%%%%%%%%%%%%%%%%%%%%%%%%%%%%%%%%%%%%%%%%%%%%%%%%%%%%%%%%%%%%%%%

\section{Counterexamples}\label{Se}

%%%%%%%%%%%%%%%%%%%%

\subsection*{The injectivity hypothesis in Corollary~\ref{Cor1}}\label{Se5}

The next example shows that we cannot omit the hypothesis that $f$ is
injective in Corollary~\ref{Cor1}.

\begin{example}\label{ExInjCor2}
Consider $X = (\Z \times \{0\}) \cup (\N \times \N)$ and $\cB = \cP(X)$.
The bimeasurable map $f : X \to X$ is given by
$$
f\big((i,0)\big) = (i+1,0) \ \text{ and } \
f\big((n,j)\big) = (n,j-1) \ \ \ \ \ (i \in \Z, n,j \in \N).
$$
The measure $\mu : \cB \to [0,\infty)$ is defined by
$$
\mu\big(\{(i,0)\}\big) = \frac{1}{2^{|i|}} \ \text{ and } \
\mu\big(\{(n,j)\}\big) =
\begin{cases}
\frac{1}{2^{n-j}} & \text{ if } 1 \leq j < n\\
\ \ 1             & \text{ if } j \geq n
\end{cases}
\ \ \ \ (i \in \Z, n,j \in \N).
$$
Since $\frac{1}{4} \mu(B) \leq \mu\big(f(B)\big) \leq 2 \mu(B)$ for every
$B \in \cB$, $f$ is bi-Lipschitz with respect to $\mu$. If $B = \{(0,0)\}$,
then conditions (i) and (ii) hold. Nevertheless, if $B_i \in \cB$ is
nonempty and satisfies condition (A) of Theorem~\ref{thm2}, then
$B_i \subset \{(k,0) : k \leq 0\}$, and so
$$
\sup_{n \in \N} \frac{\mu\big(f^{-n}(B_i)\big)}{\mu(B_i)} = \frac{1}{2}\cdot
$$
Thus, by Theorem~\ref{thm2}, $T_f$ is not Li-Yorke chaotic.
\end{example}

%%%%%%%%%%%%%%%%%%%%

\subsection*{The converse of Corollary \ref{Cor1} is false}\label{Se3}

In this subsection we assume that $f$ is surjective and bi-Lipschitz with
respect to $\mu$. If $A \in \cB$ and $0 < \mu(A) < \infty$, then
$0 < \mu\big(f^l(A)\big) < \infty$ for all $l \in \Z$, and so we can define
\begin{align*}
\q{f}{A} &= \sup\left\{\frac{\mu\big(f^k(A)\big)}{\mu\big(f^l(A)\big)}
                        : k < l \right\},\\
\qpos{f}{A} &= \sup\left\{\frac{\mu\big(f^k(A)\big)}{\mu\big(f^l(A)\big)}
                        : k < l, \ l \ge 0\right\},\\
\qneg{f}{A} &= \sup\left\{\frac{\mu\big(f^{k}(A)\big)}{\mu\big(f^{l}(A)\big)}
                        : k < l < 0 \right\}.
\end{align*}

\begin{lemma}\label{union}
If $A_1, A_2 \in \cB$ are disjoint sets of finite positive $\mu$-measure
such that $\q{f}{A_i} < \infty$ for each $i \in \{1,2\}$, then
$\q{f}{A_1 \cup A_2}< \infty$. The same holds for corresponding statements
for ${\mathcal Q_{+}}$ and ${\mathcal Q_{-}}$.
\end{lemma}

\begin{proof}
This follows easily from the subadditivity of $\mu$.
\end{proof}

\begin{lemma}\label{almostinc}
If $A \in \cB$ is a set of finite positive $\mu$-measure such that
$\mu\big(f^{i}(A)\big) \le \mu\big(f^{i+1}(A)\big)$ for all but finitely
many $i \in \Z$, then $\q{f}{A} < \infty$. The same holds for corresponding
statements for ${\mathcal Q_{+}}$ and ${\mathcal Q_{-}}$.
\end{lemma}

\begin{proof}
Let $N \in \N$ be such that
$\mu\big(f^{i}(A)\big) \le \mu\big(f^{i+1}(A)\big)$
whenever $i \in \Z$ and $|i| \ge N$.
Then, $\q{f}{A} \le \max\left(\left\{\frac{\mu(f^k(A))}{\mu(f^l(A))}
: k < l, k,l \in [-N,N] \right\} \cup \{1\}\right)$.
\end{proof}

The next example shows that the converse of Corollary~\ref{Cor1} is false
in general.

\begin{example}\label{InfMeasConvCor2}
Let $X = \N \times \Z$ and $\cB = \cP(X)$. Let $f: X \to X$ be the bijective
bimeasu\-rable map defined by
$$
f(i,j) = (i,j-1).
$$
Let
$$
X_i = \{i\} \times \Z, \  
D_i = \{i\} \times \{1,\ldots,i\}, \ 
P_i = \{i\} \times \{2i+1,\ldots,4i\}, \ 
G_i = \{i\} \times \{4i+1,\ldots\},
$$
for each $i \in \N$, and set
$$
D = \bigcup_{i=1}^{\infty} D_i, \ \ 
P = \bigcup_{i=1}^{\infty} P_i, \ \ 
G = \bigcup_{i=1}^{\infty} G_i.
$$
We define $\mu$ on $\cB$ by
\begin{equation*}
\mu\big(\{(i,j)\}\big) =
\begin{cases}
2^{-j}    & \text{ if } j \le 0 \\
2^j       & \text{ if } 1 \le j \le i\\
2^{2i-j}  & \text{ if } i+1 \le j \le 2i\\
1         & \text{ if } 2i+1  \le j \le 4i\\
2^{-j+4i} & \text{ if } j \ge 4i+1. \\
\end{cases}
\end{equation*}
We note that, for all $i \in \N$, we have that
\begin{equation*}
\mu\big(f\big(\{(i,j)\}\big)\big) =
\begin{cases}
1/2 \cdot \mu\big(\{(i,j)\}\big) & \text{ if } (i,j) \in D_i \\
\mu\big(\{(i,j)\}\big) & \text{ if } (i,j) \in P_i \\
2 \cdot \mu\big(\{(i,j)\}\big) & \text{ otherwise.}  \\
\end{cases}
\end{equation*}
In particular, $f$ is bi-Lipschitz with respect to $\mu$. Moreover,
if $A \cap D = \emptyset$, then $\mu(A) \le \mu\big(f(A)\big)$.
Now, we will establish the desired properties in a series of steps.

\medskip
\noindent {\bf Step 1.} $T_f$ is Li-Yorke chaotic.

\medskip
This follows from applying Theorem~\ref{thm2} to the sets $B_i = \{(i,0)\}$,
$i \in \N$.

\medskip
\noindent {\bf Step 2.} If $A \subset X$ is non-empty and finite, then
$\q{f}{A} < \infty$.

\medskip
Indeed, for every sufficiently large $i$,
$f^{-i}(A) \cap D = \emptyset$ and $f^{i}(A) \cap D = \emptyset$,
and so
$\mu\big(f^{-i}(A)\big) \le \mu\big(f^{-i+1}(A)\big)$ and
$\mu\big(f^{i}(A)\big) \le \mu\big(f^{i+1}(A)\big)$.
Hence, the result follows from Lemma~\ref{almostinc}.

\medskip
\noindent {\bf Step 3.} Fix $i \in \N$. If $A \subset G_i$ is nonempty,
then $\q{f}{A} < \infty.$

\medskip
Note that $\mu\big(f^{j}(A)\big) = 1/2 \cdot \mu\big(f^{j+1}(A)\big)
< \mu\big(f^{j+1}(A)\big)$ for all $j < 0$. Let $k \in \N$ be the
least such that $(i,k) \in A$. Let $A' = A \backslash \{(i,k)\}$. Then,
for $j > i+1 + k$, we have that
\begin{align*}
\mu\big(f^{j}(A)\big)
&= 2^{j-k} + \mu\big(f^j(A') \backslash D_i\big) +\ \mu\big(f^j(A) \cap D_i\big)\\
&< 2^{j-k} + \mu\big(f(f^j(A') \backslash D_i)\big) + 2^{i+1}\\
&< 2^{j-k} + \mu\big(f^{j+1}(A')\big) + 2^{j-k}\\
&= 2^{j+1-k} + \mu\big(f^{j+1}(A')\big) = \mu\big(f^{j+1}(A)\big).
\end{align*}
Now, by Lemma~\ref{almostinc}, we have that $\q{f}{A} < \infty.$

\medskip
\noindent {\bf Step 4.} For every $A \subset X$ with $0 < \mu(A) < \infty$,
we have that $\q{f}{A} < \infty$.

\medskip
As $\mu(A) < \infty$, we have that $A \backslash G$ is finite. By Step~2,
$\q{f}{A \backslash G} < \infty$ if $A \backslash G$ is nonempty. Hence,
in light of Lemma~\ref{union}, it suffices to prove the result in the case
that $A \subset G$. We will further trim $A$. Let $i \in \N$ be the least
integer such that $A \cap X_i \neq \emptyset$. Let $k \in \N$ be the least
integer such that $(i,k) \in A$. Let $E = \cup_{l=i+1}^k (A \cap X_l)$. By
Step~3 and Lemma~\ref{union}, $\q{f}{E} < \infty$ provided $E$ is nonempty.
Hence, we only need to prove the result for $A \backslash E$. Therefore,
we assume that $A \subset G$ is nonempty, $(i,k)$ is as above and, for
$i+1 \le l \le k$, we have that $A \cap X_l = \emptyset$.

For all $j < 0$, we have that $\mu\big(f^{j}(A)\big) < \mu\big(f^{j+1}(A)\big)$
as $A \subset G$. In light of Lemma~\ref{almostinc}, it will suffice to show
that $\mu\big(f^{j}(A)\big) \le \mu\big(f^{j+1}(A)\big)$ for $j > i+1+k$
to complete the proof. Fix $j > i+1+k$. We let
$I = \{m \in \N : f^j(A) \cap D_m \neq \emptyset\}$.
Note that $I$ is finite. If $I$ is empty, then by the definition of $\mu$
we have that $\mu\big(f^{j}(A)\big) \le \mu\big(f^{j+1}(A)\big)$
and we are done. Hence, assume that $I$ is nonempty and let $l = \max I$.
Note that either $l = i$ or $l > k$. In the case that $l > k$, we have that
$j > 3l > 2l +k$. Let $A' = A \backslash \{(i,k)\}$. If $l = i$, then
$$
\mu\big(f^j(A') \cap D\big) \le \mu(D_i) < 2^{i+1} < 2^{j-k}.
$$
If $l > k$, then
$$
\mu\big(f^j(A') \cap D\big) \le \sum_{m \in I} \mu(D_m)
  < \sum_{m \in I} 2^{m+1} < 2^{l+2} < 2^{j-k}.
$$
Now to conclude the proof,
\begin{align*}
\mu\big(f^{j}(A)\big)
&= 2^{j-k} + \mu\big(f^j(A') \backslash D\big) + \mu\big(f^j(A') \cap D\big)\\
&< 2^{j-k} + \mu\big(f(f^j(A') \backslash D)\big) + 2^{j-k}\\
&\le 2^{j+1-k} + \mu\big(f^{j+1}(A')\big) = \mu\big(f^{j+1}(A)\big).
\end{align*}
\end{example}

It was observed in the Introduction that the converse of Corollary~\ref{Cor1}
holds if $\mu$ is finite. However, the next example shows that if we remove
the injectivity hypothesis, then this converse may fail even for $\mu$ finite.

\begin{example}
Let $X = \N \times \N$ and $\cB = \cP(X)$. Let $f: X \to X$ be the surjective
bimeasu\-rable map defined by
\begin{equation*}
f(i,j) = \begin{cases}
         (i,j-1) & \text{ if } j > 1 \\
         (i,1)   & \text{ if } j = 1. \\
         \end{cases}
\end{equation*}
For each $i \in \N$, let
$$
X_i = \{i\} \times \N, \ 
F_i = \{(i,1)\}, \ 
D_i = \{i\} \times \{2, \ldots, i\}, \ 
G_i = X_i \backslash (F_i \cup D_i).
$$
Let
$$
F = \bigcup_{i=1}^{\infty} F_i, \ \
D = \bigcup_{i=1}^{\infty} D_i, \ \
G = \bigcup_{i=1}^{\infty} G_i.
$$
Let $\mu_i$ be the finite measure on $X_i$ so that when the points of $X_i$
are ordered in the usual fashion, their corresponding measures follow the
sequence
$$
1, 2, \ldots, 2^{i-1}, 2^{i}, 2^{i-1}, \ldots, 2, 1, 1/2, 1/4, 1/8,\ldots.
$$
In particular, note that
\begin{equation*}
\mu_i\big(f^{-1}\big(\{(i,j)\}\big)\big) =
  \begin{cases}
  2   \cdot \mu_i\big(\{(i,j)\}\big) & \text{ if } (i,j) \in D_i\\
  1/2 \cdot \mu_i\big(\{(i,j)\}\big) & \text{ if } (i,j) \in G_i.\\
  \end{cases}
\end{equation*}
Let $(\delta_i)_{i \in \N}$ be a sequence of positive numbers so that
$\sum_{i \in \N} \delta_i \mu_i(X_i) < \infty$. Define a finite measure
$\mu$ on $\cB$ by $\mu(A) = \sum_{i \in \N} \delta_i \mu_i(A \cap X_i)$
whenever $A \in \cB$. Now, we will establish the desired properties in a
series of steps.

\medskip
\noindent {\bf Step 1.} $T_f$ is Li-Yorke chaotic.

\medskip
This follows from applying Theorem~\ref{thm2} to the sets $B_i = \{(i,2)\}$,
$i \in \N$.

\medskip
\noindent {\bf Step 2.} If $A \subset D_i$ is non-empty, then
$\qneg{f}{A} < \infty.$

\medskip
This simply follows from the fact that the sequence
$\big(\mu_i\big(f^{-k}(A)\big)\big)_{k \in \N}$ is eventually decreasing
and Lemma~\ref{almostinc}.

\medskip
\noindent {\bf Step 3.} If $A \subset X$ and $A \cap F \neq \emptyset$, then
$\q{f}{A} < \infty$.

\medskip
This follows from the fact that $\mu$ is finite and $F$ is the set of
fixed points of $f$.

\medskip
\noindent {\bf Step 4.} If $A \subset X$ is non-empty, then
$\qpos{f}{A}< \infty$.

\medskip
This simply follows from the fact that $\mu$ is finite and
$\liminf_{l \to \infty} \mu\big(f^l(A)\big) > 0$.

\medskip
\noindent {\bf Step 5.} If $A \subset G$ is non-empty, then
$\qneg{f}{A} < \infty$.

\medskip
This simply follows from the fact that $\mu\big(f^{-1}(A)\big) = 1/2\cdot \mu(A)$
for any set $A \subset G$.

\medskip
\noindent {\bf Step 6.} Suppose $1 \le i < j$. Then, there exists
$L_{i,j} > 1$ such that for all $A_i \subset D_i$, $A_j \subset D_j$,
$A_i \neq \emptyset$ and $k > 0$, we have that
$$
\mu_j\big(f^{-k}(A_j)\big) \le L_{i,j} \cdot \mu_i\big(f^{-k}(A_i)\big).
$$

\medskip
As $A_i \neq \emptyset$ and $A_i \subset D_i$, we have that
$2^{-k} < \mu_i\big(f^{-k}(A_i)\big)$ for all $k > 0$.
Now, let us consider $\mu_j\big(f^{-k}(A_j)\big)$. For $0 < k \le 2j$,
we have that
$$
\mu_j\big(f^{-k}(A_j)\big) \le \mu_j\big(f^{-k}(D_j)\big)
\le (j-1) \cdot 2^{j} < j \cdot 2^{j} \cdot 2^{2j} \cdot 2^{-k}
= j \cdot 2^{3j} \cdot 2^{-k}.
$$
For $k \ge 2j$, we have that
$$
\mu_j\big(f^{-k}(A_j)\big) \le \mu_j\big(f^{-k}(D_j)\big)
\le (j-1) \cdot 2^{-k +2j} < j \cdot 2^{2j} \cdot 2^{-k}.
$$
Hence, $\mu_j\big(f^{-k}(A_j)\big) < j \cdot 2^{3j} \cdot 2^{-k}$ for all
$k > 0$. Letting $L_{i,j} = j \cdot 2^{3j}$, the result follows.

\medskip
\noindent {\bf Step 7.} Suppose that $(\delta_i)_{i \in \N}$ satisfies the
following additional property:
$\forall j \ge 2, \delta_j < 2^{-j} \cdot \max
\left\{\frac{\delta_i}{L_{i,j}} : 1 \le i \le j-1\right\}.$ Then,
$\q{f}{A} < \infty$ whenever $A \subset X$ and $\mu(A) > 0$.

\medskip
Let $A \subset X$ with $\mu(A) > 0$. By Steps 3, 4 and 5, we have that
$\q{f}{A \cap F}$ and $\q{f}{A \cap G}$ are finite, provided $A \cap F$
and $A \cap G$ are nonempty.  In light of Lemma~\ref{union}, we only need
to show that $\q{f}{A \cap D}$ is finite. Hence, let us assume $A \subset D$.
By Step~4, we only need to show that $\qneg{f}{A} < \infty$.
Let $A_i  = A \cap D_i $ for all $i \ge 1$. Fix $i$ to be the least positive
integer for which $A_i \neq \emptyset$. Let $0 < l < k$. Then, by Step~6,
\begin{align*}
\frac{\mu\big(f^{-k}(A)\big)}{\mu\big(f^{-l}(A)\big)}
&=   \frac{\delta_i \mu_i\big(f^{-k}(A_i)\big) + \sum_{j=i+1}^{\infty}
           \delta_j \mu_j\big(f^{-k}(A_j)\big)}
          {\delta_i \mu_i\big(f^{-l}(A_i)\big) + \sum_{j=i+1}^{\infty}
           \delta_j \mu_j\big(f^{-l}(A_j)\big)} \\
&\le \frac{\delta_i \mu_i\big(f^{-k}(A_i)\big) + \sum_{j=i+1}^{\infty}
           \delta_j \cdot L_{i,j} \cdot \mu_i\big(f^{-k}(A_i)\big)}
          {\delta_i \mu_i\big(f^{-l}(A_i)\big)}\\
&\le \frac{\delta_i \mu_i\big(f^{-k}(A_i)\big) + \sum_{j=i+1}^{\infty}
           2^{-j} \cdot \delta_i \cdot \mu_i\big(f^{-k}(A_i)\big)}
          {\delta_i \mu_i\big(f^{-l}(A_i)\big)}\\
&\le 2 \cdot \frac{\mu_i\big(f^{-k}(A_i)\big)}{\mu_i\big(f^{-l}(A_i)\big)}\cdot
\end{align*}
As $\qneg{f}{A_i} < \infty$ (Step~2), we have that $\qneg{f}{A} <\infty$,
completing the proof.
\end{example}

%%%%%%%%%%%%%%%%%%%%

\subsection*{The hypothesis that the measure is finite in Theorem~\ref{thm1}}
\label{Se1}
Our goal in this subsection is to show that the hypothesis that $\mu$ is
finite is essential in Theorem \ref{thm1} and Corollary \ref{cthm1}. In all
examples in this subsection we will consider $X = \Z$, $\cB = \cP(\Z)$ and
$f : i \in \Z \mapsto i+1 \in \Z$. Note that $f$ is a bimeasurable bijection.
The measure $\mu$ will be given by its values at the points of $\Z$:
$$
\mu_i = \mu(\{i\}) \ \ \ \ \ (i \in \Z).
$$
In all examples $\mu$ will be $\sigma$-finite and $f$ will be bi-Lipschitz
with respect to $\mu$.

\begin{example}\label{notfinite}
A composition operator $T_f$ satisfying only $(LY2)$ and $(LY3)$ out of the
seven conditions in Theorem~\ref{thm1}: Let
$$
\mu_i =
\begin{cases}
\dfrac{1}{2^{-i}} & \quad\textrm{if}\quad i\le -1\\[0.1in] \,\,\,1& \quad\textrm{if}\quad i\ge 0
\end{cases}.
$$
Since $\mu(B) \leq \mu\big(f(B)\big) \leq 2\mu(B)$ for every $B \in \cB$,
$f$ is bi-Lipschitz with respect to $\mu$. Moreover, for any finite set
$B \subset \Z$, $\lim_{k \to \infty} \mu\big(f^{-k}(B)\big) = 0$, proving
that $({LY3})$ is true. By Remark~\ref{rthm1}, $(LY2)$ is also true.
Now, let $\varphi\in L^p(\Z,\cP(\Z),\mu)$ be arbitrary. Then
$$
\sum_{i=-\infty}^{-1} \left|\varphi(i)\right|^p \frac{1}{2^{-i}}
  + \sum_{i=0}^\infty \left|\varphi(i)\right|^p =\Vert\varphi\Vert^p < \infty.
$$
For each $k \geq 1$, let $j_k = k/2$ if $k$ is even and $j_k = (k-1)/2$
if $k$ is odd. Then
\begin{eqnarray*}
\|T_f^k\varphi\|^p
  &=& \sum_{i \in \Z} \left|\varphi(i+k)\right|^p \mu\left(\{i\}\right)\\
  &=& \sum_{i=-\infty}^{-1} \left|\varphi(i+k)\right\vert^p \frac{1}{2^{-i}}
      + \sum_{i=0}^\infty \left|\varphi(i+k)\right|^p\\
  &=& \frac{1}{2^k}\sum_{i=-\infty}^{k-1} \left|\varphi(i)\right\vert^p
      \frac{1}{2^{-i}} + \sum_{i=k}^\infty \left|\varphi(i)\right|^p\\
  &\leq& \frac{1}{2^k}\sum_{i=-\infty}^{-1} \left|\varphi(i)\right\vert^p
      \frac{1}{2^{-i}} + \frac{1}{2^{k-j_k+1}} \sum_{i=0}^{j_k-1}
      \left|\varphi(i)\right\vert^p + \sum_{i=j_k}^\infty
      \left|\varphi(i)\right|^p\\
  &\to& 0 \ \text{ as } k \to \infty.
\end{eqnarray*}
Thus, all orbits under $T_f$ converge to zero. In particular, $(LY1)$ is
false. By Remark~\ref{rthm1}, $(LY6)$ and $(LY7)$ are also false.
Moreover, by the definition of $\mu$, for any non-empty set $B \subset \Z$,
$\liminf_{k\to\infty} \mu\big(f^k(B)\big) \geq 1$, proving that
$(LY4)$ and $(LY5)$ are also false.
\end{example}

\begin{example}\label{inli}
A composition operator $T_f$ such that the equivalent properties
$(LY6)$ and $(LY7)$ hold, but $(LY4)$ does not (in particular,
$T_f$ is Li-Yorke chaotic but $(T_f)^{-1} = T_{f^{-1}}$ is not):
Consider
$$
\ldots \mu_{-3} \mu_{-2}, \mu_{-1} \,\Big|\, \mu_{0} , \mu_{1} ,\mu_{2}\ldots
= \ldots 1, \frac12, \frac{1}{2^2}, \frac12, 1, 1, \frac12, 1 \,\Big| \,
  1, 1, 1 \ldots,
$$
where the bar indicates that $\mu_{-2} = \frac12$, $\mu_{-1} = 1$,
$\mu_0 = 1$, $\mu_1 = 1$, $\mu_2 = 1$, and so on,
and in the left hand side we have successive blocks of the form
$$
1,\frac12,\frac{1}{2^2},\ldots,\frac{1}{2^{k-1}},\frac{1}{2^k},
\frac{1}{2^{k-1}},\ldots,\frac{1}{2^2},\frac12,1.
$$
Since $\frac12\mu(B) \le \mu\big(f(B)\big) \le 2\mu(B)$ for every $B \in \cB$,
$f$ is bi-Lipschitz with respect to $\mu$. In particular, $T_f$ and
$(T_f)^{-1} = T_{f^{-1}}$ are continuous linear operators acting on
$L^p\big(\Z,\cP(\Z),\mu\big)$. Since $\mu(\{0\}) = 1 > 0$,
$ \liminf_{k\to\infty} \mu\big(f^{-k}(\{0\})\big) = 0$ and
$\limsup_{k\to\infty} \mu\big(f^{-k}(\{0\})\big) > 0$,
$(LY6)$ is true. By Remark \ref{rthm1}, $T_f$ is Li-Yorke chaotic.
As for $g = f^{-1}$, note that for any non-empty set $B$, we have that
$$
\liminf_{k\to\infty} \mu\big(g^{-k}(B)\big) =
\liminf_{k\to\infty} \mu\big(f^{k}(B)\big) \geq 1.
$$
Hence, $(LY3)$ is false for $g$. By Remark \ref{rthm1}, $(T_f)^{-1}$ is not
Li-Yorke chaotic.
\end{example}

\begin{example}
A composition operator $T_f$ such that $(LY1)$ holds, but $(LY4)$ and
$(LY7)$ do not: Consider
$$
\ldots \mu_{-3} \mu_{-2}, \mu_{-1} \,\Big|\, \mu_{0} , \mu_{1} ,\mu_{2}\ldots
= \ldots \frac{1}{2^3}, \frac{1}{2^2}, \frac{1}{2} \, \Big| \,
   1, 2, 1, 1, 2, 2^2, 2, 1 \ldots,
$$
where in the right hand side we have successive blocks of the form
$$
1,2,2^2,\ldots,2^{k-1},2^k,2^{k-1},\ldots,2^2,2,1.
$$
Since $\frac12\mu(B) \le \mu\big(f(B)\big) \le 2\mu(B)$ for every $B \in \cB$,
$f$ is bi-Lipschitz with respect to $\mu$. By Corollary~\ref{cthm2},
$T_f$ satisfies $(LY1)$. If a characteristic function $\chi_B$ lies in
$L^p(\Z,\cP(\Z),\mu)$, then $B \cap \N$ is finite, and so $T_f^n(\rchi_B) \to 0$
as $n \to \infty$. This shows that $(LY7)$ is false. It is also clear
that $(LY4)$ fails.
\end{example}

\begin{example}
A composition operator $T_f$ satisfying only $(LY4)$ out of the seven
conditions in Theorem~\ref{thm1}: It is enough to define
$$
\ldots \mu_{-3} \mu_{-2}, \mu_{-1} \,\Big|\, \mu_{0} , \mu_{1} ,\mu_{2}\ldots
= \ldots 1,1,1 \, \Big| \, \frac{1}{2},\frac{1}{2^2},\frac{1}{2^3} \ldots.
$$
\end{example}

%%%%%%%%%%%%%%%%%%%%

\subsection*{The injectivity hypothesis in Theorem~\ref{thm1}}\label{Se2}

The hypothesis that $f$ is injective in Theorem \ref{thm1} cannot be removed.
In fact, as the next example shows, fail of the injectivity of $f$ at only
$2$ points may prevent $T_f$ from being Li-Yorke chaotic.

\begin{example}\label{ninjective}
Let $X = \{0\} \cup \left\{\frac1i : i \ge 1\right\}$ and $\cB = \cP(X)$.
The finite measure $\mu$ is defined by its values at the elements of $X$
as follows:
$$
\mu\left(\{0\}\right) = 0 \ \ \text{ and } \ \
\mu\left(\left\{\frac{1}{i}\right\}\right) =\; \frac{1}{2^i} \ \
\text{ for } i \geq 1.
$$
The map $f : X \to X$ is given by
$$
f(0) = 0,\quad
f\left(\frac{1}{i}\right) = \dfrac{1}{i-1} \ \ \text{ for } i \geq 2, \quad
f(1) = 1.
$$
Clearly $f$ is surjective, but it is not injective since
$f\left(\frac12\right) = f(1)$. We claim that
$$
\frac12 \mu(B) \leq \mu\big(f(B)\big) \leq 2 \mu(B) \ \ \text{ for every }
B \in \cB.
$$
Indeed, since the second inequality is clear, let us prove the first one. If
$\big\{\frac{1}{2},1\big\} \not\subset B$, then $f$ is injective on $B$ and
$\mu\big(f(B)\big) \ge \mu(B)$. Assume $\left\{\frac12,1\right\} \subset B$
and write $B$ as the union $B = \big\{\frac12,1\big\} \cup B'$,
where $\left\{\frac12,1\right\} \cap B' = \emptyset$. Since
$\mu\big(f(B')\big) \ge \mu(B')$, we obtain
$$
\frac{\mu\big(f(B)\big)}{\mu(B)}
  = \frac{\frac12+\mu(B')}{\frac14+\frac12+\mu(B')}
\ge \frac{\frac12}{\frac14+\frac12+\frac14}
  = \frac12\cdot
$$
Thus, $f$ is bi-Lipschitz with respect to $\mu$.
If $B = \left\{\frac12\right\}$, then $\mu(B) = \frac14 > 0$ and
$$
\lim_{k \to \infty} \mu\big(f^{-k}(B)\big)
= \lim_{k \to \infty}\frac{1}{2^{k+2}} = 0.
$$
Hence, $(LY3)$ is true and, by Remark \ref{rthm1}, so is $(LY2)$.
Now, let $\varphi\in L^p(X,\cB,\mu)$ be arbitrary. Note that
$$
\left(T_f^k \varphi\right)\left(\frac1i\right)
= \varphi\left(f^k\left(\frac1i\right)\right)
= \begin{cases} \varphi(1) & \textrm{if} \quad 1\le i\le k \\
  \varphi\left(\frac{1}{i-k}\right) & \textrm{if} \quad i\ge  k+1
\end{cases}.
$$
Hence,
\begin{eqnarray*}
\|T_f^k\varphi\|^p
&=& \sum_{i=1}^\infty \left|\left(T_f^k\varphi\right)\left(\frac1i\right)
    \right|^p \mu\left(\left\{\frac1i\right\}\right)\\
&=& |\varphi(1)|^p \sum_{i=1}^{k} \frac{1}{2^i} + \sum_{i=k+1}^\infty
    \left|\varphi\left(\frac{1}{i-k}\right)\right|^p \frac{1}{2^i}\cdot
\end{eqnarray*}
Therefore, if $\varphi(1)\neq 0$, then $\|T_f^k\varphi\|^p \ge \frac12
\left|\varphi(1)\right|^p > 0$ for all $k \ge 1$, implying that $\varphi$
is not a semi-irregular vector for $T_f$. Suppose that $\varphi(1) = 0$.
Then
$$
\|T_f^k\varphi\|^p = \frac{1}{2^k} \sum_{i=k+1}^\infty
  \left|\varphi\left(\frac{1}{i-k}\right)\right|^p \frac{1}{2^{i-k}}
  = \frac{1}{2^k} \Vert \varphi\Vert^p \to 0 \text{ as } k \to \infty,
$$
showing that $\varphi$ is not a semi-irregular vector of $T_f$. Therefore,
$(LY1)$ is false. By Remark~\ref{rthm1}, $(LY6)$ and $(LY7)$ are also false.
Finally, note that any set $B \subset X$ of positive measure contains a point
of the form $\frac1i$, for some $i \ge 1$. As $f^k\left(\frac1i\right) = 1$
for every $k$ big enough, we conclude that $\mu\left(f^k(B)\right) \ge
\frac12$ for any such $k$, proving that $(LY4)$ (hence $(LY5)$) is false.
\end{example}

%%%%%%%%%%%%%%%%%%%%

\subsection*{A generically Li-Yorke chaotic composition operator that is not
topologically transitive}\label{Se4}

It was obtained in \cite[Theorem~3.13]{Pra} and \cite[Theorem~37]{BBMP2}
examples of unilateral weighted forward shifts
$$
T : (x_1,x_2,x_3,\ldots) \in \ell^2 \mapsto
    (0,w_1x_1,w_2x_2,w_3x_3,\ldots) \in \ell^2
$$
that are generically Li-Yorke chaotic. Moreover, the weights $w_j$ satisfy
$1/4 \leq w_j \leq 2$ for all $j \in \N$. Clearly, such an operator cannot
be topologically transitive. Let us see that such an operator $T$ can be
regarded as a composition operator $T_f$ on $L^2(\N,\cP(\N),\mu)$ for
suitable $\mu$ and $f$. Indeed, the measure $\mu$ is defined by
$$
\mu(\{1\}) = \infty, \ \ \mu(\{2\}) = 1 \ \ \text{ and } \ \
\mu(\{i\}) = (w_1 \cdot\ldots\cdot w_{i-2})^2 \ \text{ for } i \geq 3,
$$
and the bimeasurable map $f : \N \to \N$ is given by
$$
f(1) = 1 \ \ \text{ and } \ \ f(i) = i-1 \ \text{ for } i \geq 2.
$$
An easy computation shows that (\ref{condition}) holds with $c = 1/4$. Hence,
the composition operator $T_f$ is well-defined on $L^2(\N,\cP(\N),\mu)$.
Now, let
$$
\phi : (x_n)_{n \in \N} \in L^2(\N,\cP(\N),\mu) \mapsto
       (x_2,w_1x_3,w_1w_2x_4,w_1w_2w_3x_5,\ldots) \in \ell^2.
$$
Then, $\phi$ is an isometric isomorphism from $L^2(\N,\cP(\N),\mu)$ onto
$\ell^2$ and $T \circ \phi = \phi \circ T_f$.
This shows that $T$ is topologically conjugate to $T_f$. Thus, from a
dynamical systems point of view, we can regard $T$ as being $T_f$.

%%%%%%%%%%%%%%%%%%%%%%%%%%%%%%%%%%%%%%%%%%%%%%%%%%%%%%%%%%%%%%%%%%%%%%%%%%%%%

\section*{Acknowledgements}

 The first, the second and the third authors were partially supported by
grants {\#}2017/22588-0, {\#}2017/19360-8 and {\#}2018/06916-0,
S\~ao Paulo Research Foundation (FAPESP), respectively.
The first and the third authors were also partially supported by CNPq.

%%%%%%%%%%%%%%%%%%%%%%%%%%%%%%%%%%%%%%%%%%%%%%%%%%%%%%%%%%%%%%%%%%%%%%%%%%%%%


\begin{thebibliography}{99}

\bibitem{BDP}
    F. Bayart, U. B. Darji and B. Pires,
    {\it Topological transitivity and mixing of composition operators},
    J. Math.\ Anal.\ Appl.\ {\bf 465} (2018), no.\ 1, 125--139.

\bibitem{BM}
    F. Bayart and \'E. Matheron,
    {\it Dynamics of Linear Operators},
    Cambridge University Press, Cambridge, 2009.

\bibitem{NBBMP}
    T. Berm\'udez, A. Bonilla, F. Mart\'inez-Gim\'enez and A. Peris,
    {\it Li-Yorke and distributionally chaotic operators},
    J. Math.\ Anal.\ Appl.\ {\bf 373} (2011), no.\ 1, 83--93.

\bibitem{BBMP}
    N. C. Bernardes Jr., A. Bonilla, V. M\"uller and A. Peris,
    {\it Distributional chaos for linear operators},
    J. Funct.\ Anal.\ {\bf 265} (2013), no.\ 9, 2143--2163.

\bibitem{BBMP2}
    N. C. Bernardes Jr., A. Bonilla, V. M\"uller and A. Peris,
    {\it Li-Yorke chaos in linear dynamics},
    Ergodic Theory Dynam.\ Systems {\bf 35} (2015), no.\ 6, 1723--1745.

\bibitem{BCDMP}
    N. C. Bernardes Jr., P. R. Cirilo, U. B. Darji, A. Messaoudi and
    E. R. Pujals,
    {\it Expansivity and shadowing in linear dynamics},
    J. Math.\ Anal.\ Appl.\ {\bf 461} (2018), no.\ 1, 796--816.

\bibitem{BMPP}
     J. B\`es, Q. Menet, A. Peris and Y. Puig,
     {\it Recurrence properties of hypercyclic operators},
     Math.\ Ann.\ {\bf 366} (2016), no.\ 1, 545--572.

\bibitem{GMM}
     S. Grivaux, \'E. Matheron and Q. Menet,
     {\it Linear dynamical systems on Hilbert spaces: typical properties
          and explicity examples},
     preprint (arXiv:1703.01854).

\bibitem{grosse}
    K.-G. Grosse-Erdmann and A. Peris Manguillot,
    {\it Linear Chaos}, Springer, London, 2011.

\bibitem{HaKa}
    A. B. Hajian and S. Kakutani,
    {\it Weakly wandering sets and invariant measures},
    Trans.\ Amer.\ Math.\ Soc.\ {\bf 110} (1964), 136--151.

\bibitem{LY70}
     T. Y. Li and J. A. Yorke,
     {\it Period three implies chaos},
     Amer.\ Math.\ Monthly {\bf 82} (1975), no.\ 10, 985--992.

\bibitem{M}
    Q. Menet,
    {\it Linear chaos and frequent hypercyclicity},
    Trans.\ Amer.\ Math.\ Soc.\ {\bf 369} (2017), no.\ 7, 4977--4994.

\bibitem{Pra}
    G. T. Pr\v{a}jitur\v{a},
    {\it Irregular vectors of Hilbert space operators},
    J. Math.\ Anal.\ Appl.\ {\bf 354} (2009), no.\ 2, 689--697.

\bibitem{R}
    W. Rudin,
    {\it Functional Analysis}, Second Edition,
    McGraw-Hill Inc, New York, 1991.

\end{thebibliography}
\end{document}